\documentclass[psamsfonts,reqno]{amsart}

\usepackage{amsmath}
\usepackage{amsfonts}
\usepackage{amssymb}
\usepackage{amssymb,amsfonts,calligra,verbatim,hyperref}
\usepackage[all,arc]{xy}
\usepackage{enumerate}
\usepackage{mathrsfs}
\usepackage{tikz-cd}
\usepackage{mathtools}
\usepackage{comment}

\usepackage{color}   
\usepackage{hyperref}
\hypersetup{
    colorlinks=true, 
    linktoc=all,     
    linkcolor=blue,  
}

\newtheorem{thm}{Theorem}[section]

\newtheorem{prop}[thm]{Proposition}
\newtheorem{lem}[thm]{Lemma}

\theoremstyle{definition}
\newtheorem{defn}[thm]{Definition}

\theoremstyle{remark}
\newtheorem{rem}[thm]{Remark}

\newcommand{\spec}{\operatorname{Spec}}
\DeclareMathOperator{\Hom}{\mathscr{H}\text{\kern -3pt {\calligra\large om}}\,}

\begin{document}

\makeatletter
\let\c@equation\c@thm
\makeatother
\numberwithin{equation}{section}

\title{Log Fundamental Group scheme}

\author{Aritra Sen}

\date{}

\begin{abstract}

In this article, we study the log-scheme theoretic version of the Nori fundamental group scheme. Similar to the classical Nori fundamental group scheme, the log Nori fundamental group scheme classifies torsors on log flat topology. We also compute the nilpotent pro-p quotient of the the log Nori fundamental group scheme with log at $\{0, 1, \infty\}$.

\end{abstract}

\maketitle

\tableofcontents

\section{Introduction}
Let $X$ be a connected, reduced scheme over a field $k$ and $x \in X(k)$ be rational point. Then Nori proved the following result in \cite{nori1982fundamental}.
\begin{thm}
\label{NoriFundamentalgrp}
 There exists a profinite group scheme, $\pi^{N}(X,x)$, the Nori fundamental group scheme, such that for any finite group scheme over $k$ there is a functorial isomorphism between $\operatorname{Hom_{grp.schm}}(\pi^{N}(X,x), G)$ and $\operatorname{Tor}(G,(X,x))$, where $\operatorname{Tor}(G,(X,x))$ is the set of isomorphism classes of $(T,\alpha)$ where $T$ is a $G$-torsor and $\alpha$ is a trivialization $\alpha : T|_{x} \cong G$.
\end{thm}
 
 We also have following relationship

\begin{equation}
    \operatorname{Hom_{grp.schm}}( \pi^{N}(X,x), G )/ \sim \cong \operatorname{ker}( H^1_{fppf}(X,G) \to H^1_{fppf}(x,G))
\end{equation}

Here $\sim$ denotes conjugation by $G(k)$ and the kernel of the map between the pointed sets  $H^1_{fppf}(X,G) \to H^1_{fppf}(x,G)$ is the inverse image of distinguished element in the $H^1_{fppf}(x,G))$.

In this paper, we generalize the notion of Nori fundamental group schemes to log schemes. Let $X$ be a fine saturated log scheme which is log smooth over a field $k$ whose underlying scheme is connected. Then we define $\pi^{N}_{\log}(X)$ which classifies log flat torsors of $X$ over finite classical $k$-group schemes.

The following is the main theorem proved in this paper

\begin{thm}
\label{mainthm1}
Let $X$ be log smooth scheme over a field $k$ (the log structure on field $k$ is trivial) such that the underlying scheme of $X$ is connected and $x \in X(k)$. There exists a profinite group scheme, $\pi^{N}_{log}(X,x)$ such that for any finite group scheme over $k$ there is a functorial isomorphism between $Hom(\pi^{N}_{log}(X,x), G)$ and $\operatorname{Tor_{log}}(G,(X,x))$, where $\operatorname{Tor_{log}}(G,(X,x))$ is the isomorphism classes of $(T,\alpha)$ where $T$ is a $G$-sheaf torsor in the log flat topology of $X$ and $\alpha$ is a trivialization $\alpha : T|_{x} \cong G$.
\end{thm}

Similar to classical case we get the following relationship.
\begin{equation}
    \operatorname{Hom_{grp.schm}}( \pi^{N}_{log}(X,x), G )/ \sim  \cong  \operatorname{ker}( H^1_{logfl}(X,G) \to H^1_{logfl}(x,G)).
\end{equation}

In the second part of the paper we prove the following
\begin{thm}
\label{mainthm2}
Let $k$ be an algebraically closed field of characteristic $p > 0$. Let $X= \mathbb{P}^{1}_k$ and $D = \{0, 1 , \infty \}$ . We endow $X$ with the log structure associated with $D$. Then
$$(\pi^{N}_{log}(X,x))_{p-nilp} \cong \underset{n}{\varprojlim} \ \mu_{p^n}\times \mu_{p^n}$$.

Here $(\pi^{N}_{log}(X,x))_{p-nilp}$ refers to the pro-p nilpotent quotient of $\pi^{N}_{log}(X,x)$ (See Definition \autoref{nilpotent}).
\end{thm}

The original construction of fundamental group scheme by Nori was for a scheme over a field $k$.This construction was generalized to a scheme $X$ when it is a reduced flat scheme over a Dedekind scheme in \cite{gasbarri2003heights}. In \cite{borne2015nori}, the fundamental group scheme was generalized to the Nori fundamental gerbe of a fibered category. In Chapter 3 of \cite{nori1982fundamental}, it is proved that if $X$ is smooth, connected projective curve over an algebraically closed field of $k$ with a base point $x_0$ and let $S$ be a finite set of points not containing $x_0$ then the $\pi^{N}(X-S,x_0)$ is the group scheme associated with Tannakian category of parabolic sheaves on $X-S$. In \cite{borne2012parabolic}  it was shown how parabolic
sheaves on a scheme can be reinterpreted in terms of logarithmic geometry. 
\subsection{Structure of the paper}
The structure of the paper is as follows: in Sections \ref{sec2}, \ref{sec3} and \ref{sec4} we review some properties of the Nori fundamental group scheme, log flat topology and log smooth morphism respectively. In Section \ref{sec5} we discuss torsors in log flat topology. The Theorem \ref{mainthm1} follows from the Theorem \ref{torsorsfibreproduct} in Section \ref{sec6}. In Section \ref{sec7} we show that the log flat torsors of solvable group schemes are representable. The Section \ref{sec8} contains the proof of Theorem \ref{mainthm2}. We define a canonical map $f$ from $(\pi^{N}_{log}(X,x))_{p-nilp} $ to $ \underset{n}{\varprojlim} \ \mu_{p^n}\times \mu_{p^n}$. In Section \ref{sec8parta} we show that $f$ is surjective and in Section \ref{sec8partb} we show $f$ is injective.

\textit{Acknowledgement}. I would like to express my gratitude to Professor Kazuya Kato for his kind advice and feedback while writing this article. I would also like to thank Professor Madhav Nori for his advice and answering several questions which came up while writing this article. 

\section{Nori fundamental group scheme} \label{sec2}
Let $X$ be a reduced and connected scheme over a field $k$. Let $x$ be a $k$-rational point of $X$. Consider the category $\mathcal{T}(X,x)$ whose objects consists of triples of the form $(G,T,t)$ where $G$ is a finite group scheme over $k$, $T$ is a $G$-torsor in the fqpc topology of $X$, and $t$ is $k$-rational point which lies above $x$. A morphism in $\mathcal{T}(X,x)$ between $(G_1,T_1,t_1)$ to $(G_2,T_2,t_2)$ consists of a pair $g,f$ where $g$ is a group scheme homomorphism over $k$ and $f$ is a morphism of schemes from $T_1$ to $T_2$ over $X$ such that the following diagram commutes:

\begin{center}
\begin{tikzcd}
G_1 \times T_1 \arrow{r}{m_1} \arrow[swap]{d}{g \times f} &  T_1 \arrow{d}{f} \\
G_2 \times T_2 \arrow{r}{m_2} & T_2
\end{tikzcd}
\end{center}

where $m_1:G_1 \times T_1 \to T_1$ and $m_2:G_2 \times T_2 \to T_2$ is the action of $G_1$ on $T_1$ and action of $G_2$ on $T_2$ respectively.

The Nori fundamental group scheme is defined as follows

\begin{defn}
Let $(\pi^{N}(X,x),T,t)$ be such that $T$ is a $\pi^{N}(X,x)$-torsor over 
$X$ and $t$ is a $k$-rational point over $x$. If there exists a unique morphism from $
(\pi^{N}(X,x),T,t)$ to any object $(G',T',t')$ in $\mathcal{T}(X,x)$, $
\pi^{N}(X,x)$ is called the Nori fundamental group scheme of $(X,x)$.
\end{defn}

\begin{defn}
Let $(G_1,T_1,t_1)$, $(G_2,T_2,t_2)$, $(G,T,t)$ be objects in $\mathcal{T}(X,x)$. Let $(g_1,f_1)$ and $(g_2,f_2)$ be morphism from $(G_1,T_1,t_1)$ to $(G,T,t)$ and $(G_2,T_2,t_2)$ to $(G,T,t)$ respectively. If the triple $(G_1 \times_G G_2, T_1 \times_T T_2, t_1 \times_t t_2)$ is an object of $\mathcal{T}(X,x)$ for any $(G_1,T_1,t_1)$, $(G_2,T_2,t_2)$, $(G,T,t)$  in $Ob(\mathcal{T}(X,x))$ then we say the category $\mathcal{T}(X,x)$ has property $\mathcal{P}$.
\end{defn}

The following two theorems are proved in Proposition 1 and Proposition 2 of Chapter 2 in  \cite{nori1982fundamental}.
\begin{thm}
The Nori fundamental group scheme of $(X,x)$ exists iff $\mathcal{T}(X,x)$ has property $\mathcal{P}$.
\end{thm}

\begin{thm}\label{thm::nori}
If $X$ is reduced and connected then $\mathcal{T}(X,x)$ has property $\mathcal{P}$.
\end{thm}
\section{Log flat topology} \label{sec3}
For the definitions of log schemes and fine saturated log schemes (fs log schemes), see \cite{kato1989logarithmic}.

We need the following definitions to define log flat topology on a fs log scheme.

\begin{defn}
Let $X$ and $Y$ two fs log schemes. Let $f:X \to Y$ be a morphism of fs log schemes. A chart for $f$ is a triple $(a,b,h)$ where $a:X \to \spec(\mathbb{Z}[P])$ and $b:X \to \spec(\mathbb{Z}[Q]$ are charts for $X$ and $Y$ respectively, $h:Q \to P$ is a morphism of monoids such that the following diagram commutes

\begin{center}
\begin{tikzcd}
X \arrow{r}{a} \arrow[swap]{d}{f} & \spec(\mathbb{Z}[P]) \arrow{d}{h^*} \\%
Y \arrow{r}{b}& \spec(\mathbb{Z}[Q]).
\end{tikzcd}
\end{center}

Here $h^*$ is morphism induced on  $\spec(\mathbb{Z}[P])$ to $\spec(\mathbb{Z}[Q])$ by the monoid morphism $h$.
\end{defn}

\begin{defn}
Let $f: X \to Y$ be a morphism of fs log schemes. We say $f$ is a log flat morphism if there exists a chart for $f$ fppf locally $(a:X \to \spec(\mathbb{Z}[P]),b:X \to \spec(\mathbb{Z}[Q],h:Q \to P)$ such that it satisfies the following conditions are satisfied
\begin{enumerate}
\item The induced map from $X \to Y \times_{\spec(\mathbb{Z}[Q]} \spec(\mathbb{Z}[P])$ is classically flat.
\item The induced map from $Q^{\text{gp}} \to P^{\text{gp}}$ is injective.  
\end{enumerate}
\end{defn}

\begin{defn}
\begin{enumerate}
\item A morphism of monoids $h:P \to Q$ is Kummer if it is injective and there exists an $n \geq 1$ such that $Q^{n} \subset h(P)$.
\item Let $X=(X,  \mathcal{O}_X, M_X)$ and $Y=(Y, \mathcal{O}_Y, M_Y)$ be fs log schemes and $f:X \to Y$ be a morphism of fs log schemes. We say $f$ is Kummer for any $x \in X$, the induced map of monoids $(M_Y/\mathcal{O}^{*}_Y)_{f(\bar{x})} \to (M_X/\mathcal{O}^{*}_X)_{\bar{x}}$ is Kummer.
\end{enumerate}
\end{defn}

\begin{defn}
We say a morphism of fs log schemes $f:X \to Y$ is Kummer log flat if it is both Kummer and log flat.
\end{defn}

We can show that it is always possible a pick nice charts fppf locally for Kummer log flat morphism.

\begin{lem}\label{lem::chart}
Let $f : X \to Y$ be a Kummer log flat morphism of fs log schemes.
Let $x$ be a point of X with its image $y$ in Y . Then the following hold.
\begin{enumerate}
\item  There is a fppf local chart  of $f$, $(a:X \to \spec(\mathbb{Z}[P]),b:Y \to 
\spec(\mathbb{Z}[Q]),h)$ such that the monoid morphism $h:Q \to P$ is Kummer
,the induced maps $Q \to (M_Y/\mathcal{O}^{*}_Y)_{\bar{y}}$ and $P \to (M_X/\mathcal{O}^{*}_X)_{\bar{x}}$ are isomorphisms, and such that the induced
morphism of schemes $X \to Y \times^{\text{fs}}_{\spec(\mathbb{Z}[Q])} \spec(\mathbb{Z}[P])$ is strict and classically flat.

\item   Further, assume that $f$ is locally of finite presentation. Then, we can take a chart
as in (1) such that $X \to Y \times^{\text{fs}}_{\spec(\mathbb{Z}[Q])} \spec(\mathbb{Z}[P])$ is also surjective and locally of finite presentation.
\end{enumerate}
\end{lem}
\begin{proof}
See Proposition 1.3 in \cite{illusie2013log}.
\end{proof}

We are finally ready to define a Grothendieck topology on the category of fs log schemes over an fs log scheme X, $FSLSch/X$, which is called the log flat topology.

\begin{defn} 
Let $X$ be a fs log scheme. The \textit{log} \textit{flat} 
\textit{site}, $(FSLSch/X)_{logfl}$, is defined as follows. The log flat topology on $(FSLSch/X)_{logfl}$ is given by covering families $\{f_i:T_i \rightarrow T\}_{i \in I}$ of morphisms in $(FSLSch/X)$ such that $ \bigcup \limits_{i \in I} f_i(T_i) = T$ 
and $f_i$ are Kummer log flat morphisms and the underlying morphism of schemes $f_i$ is are locally of finite presentation. We will denote this site by $X_{\text{fl}}^{\text{log}}$
\end{defn} 

\begin{defn}
Let $X= (X, \mathcal{O}_X, M_X)$ and $Y = (Y, \mathcal{O}_Y, M_Y )$. Let $f:X \to Y$ be a morphism of fs log schemes. We say $f$ is strict if the induced map $f^{*}M_Y \to M_X$ is an isomorphism.
\end{defn}

Let $X$ be an fs log scheme, we denote the underlying scheme of $X$ by $|X|$. Fiber products exist in the category of fs log scheme but the classical fiber product of three fs log schemes may not be an fs log scheme.  So we denote fiber product in the category of fs log schemes it by $\times^{\text{fs}}$ to distinguish it from the classical fiber product of schemes. In general, $|X \times^{\text{fs}}_{Z} Y|$ but sometimes they can be isomorphic as illustrated by the following lemma.

\begin{lem}\label{Lem::strict}
Let $f : X \to Z$ and $g: Y \to Z$ be two morphisms of fs log schemes such that $f$ is strict. Then  $|X \times^{\text{fs}}_{Z} Y| \cong |X \times_{Z} Y|$ and the canonical map from $X \times^{\text{fs}}_{Z} Y \to Y$ is again strict.
\end{lem}

\section{Log smooth morphisms}\label{sec4}
Please see Section 3.3 in \cite{kato1989logarithmic} for definition of log smooth morphisms.

Here we will use the following characterization of log-smooth morphism
\begin{thm}
\label{logcriterion}
Let $f: X \to Y$ be a morphism of fs log schemes and $,b:Y \to 
\spec(\mathbb{Z}[Q])$ be a chart for $Y$. Then the following conditions
are equivalent:
\begin{enumerate}
\item $f$ is log smooth.
\item There exists etale locally a chart $(a:X \to \spec(\mathbb{Z}[P]),b:Y \to 
\spec(\mathbb{Z}[Q]),h)$ of $f$ such that $\operatorname{Ker}(Q^{\text{gp}} \to P^{\text{gp}})$ and the torsion part of $\operatorname{Coker}(Q^{\text{gp}} \to P^{\text{gp}})$ are finite groups of orders invertible on $X$, and the induced map from $X \to Y \times_{\spec(\mathbb{Z}[Q]} \spec(\mathbb{Z}[P])$ is classically etale.
\end{enumerate}
\end{thm}
\begin{proof}
See Theorem 3.5 in \cite{kato1989logarithmic}.
\end{proof}

\begin{defn}
Let $k$ be a field with trivial log structure. Let $X$ be an fs log scheme over $k$.
We say $X$ is log smooth over $k$ if the structure morphism from $X$ to $k$ is log smooth.
\end{defn}

\begin{lem}
\label{fibreprodstalk}
Let $X_1, X_2$ and $X_3$ be fs log schemes with morphisms $X_2 \to X_1$ and $X_2 \to X_1$  . Let $X_4 = X_3 \times_{X_1}^{\text{fs}} X_2$. Let $x \in X_4$ and
$Q_i = \bar{M}_{X_i,x}$ ( here $x \to X_4 \to X_i$ is also denoted by $x$). Then the natural map from $(Q_3 \oplus_{Q_1}^{\text{sat}} Q_2)/(Q_3 \oplus_{Q_1}^{\text{sat}} Q_2)^{\times} \to Q_4 $ is an isomorphism.
\end{lem}
\begin{proof}
See Proposition 2.1.1 in \cite{nakayama1997logarithmic}. 
\end{proof}

\begin{lem}
\label{isomonoid}
Consider the morphism of saturated monoids $Q_1 \to Q_2$ and $Q_1 \to Q_3$. Let $Q_1 \to Q_2$ be an inclusion such that $Q^n \subset P$. Every element $Q_1$ becomes $n$-divisible when mapped to $Q_3$. Then $Q_3 \to (Q_3 \oplus_{Q_1}^{\text{sat}} Q_2)/(Q_3 \oplus_{Q_1}^{\text{sat}} Q_2)^{\times}$ is an isomorphism.
\end{lem}
\begin{proof}
See Lemma 2.3 in \cite{hagihara2016structure}.
\end{proof}
The following proposition is used in the proof of \ref{torsorsfibreproduct}.

\begin{prop}\label{redcov}
Let $k$ be a field and $X$ be a fs log scheme which is log smooth over the field $k$ ( the log structure on $k$ is a trivial).
Let $f : Y \rightarrow X$ be a Kummer log-́flat morphism . Let
$y \in Y, x = f (y)$. Then there exists a commutative diagram of fs log schemes

\begin{center}
\begin{tikzcd}
T \arrow{r}{\phi} \arrow[swap]{d}{h} & X_n \arrow{d}{g} \\%
Y \arrow{r}{f}& X
\end{tikzcd}
\end{center}

such $\phi$ is a classical flat morphism and $h$ is a surjective Kummer log flat morphism. 
Moreover $X_n$ is reduced.
\end{prop}

\begin{proof}
Since $X$ is log smooth over the field $k$ by $\ref{logcriterion}$ there exist charts (etale locally) $a:X \to \spec(\mathbb{Z}[P])$ such that the canonical map from $X \to \spec(k[P])$ is etale.  Let $X_n = X \times^{\text{fs}}_{\spec(\mathbb{Z}[P])} \spec(\mathbb{Z}[P^{1/n}])$ where $P^{1/n}$ is a $P$ -monoid such that $P \to P^{1/n}$ is 
isomorphic to $n: P \to P$  $( a \to a^n)$. Note that the canonical map from $X_n \to \spec(k[P^{1/n}])$ is etale and since $\spec(k[P^{1/n}])$ is reduced it implies $X_n$ is also reduced.

Now let $T = Y \times_X^{\text{fs}} X_n$. We will have to show that the morphism $\phi:T \to X_n$ is strict. This is equivalent to show for all $x \in T$, the induced morphism $\overline{M}_{X_n,\phi(x)} \to \overline{M}_{T,x}$ is an isomorphism. Using Lemma \ref{fibreprodstalk} we can $ \overline{M}_{T,x}$ and Lemma $\ref{isomonoid}$ shows $\overline{M}_{X_n,\phi(x)} \to \overline{M}_{T,x}$ is an isomorphism. The morphism $g: X_n \to X$ is surjective and Kummer log flat therefore the morphism $h: T \to Y$ is surjective and Kummer log flat. 
\end{proof}

\section{Torsors in log flat topology}\label{sec5}
In this and the following sections  we will consider only fs log schemes over a field $k$ with the log structure on $k$ being trivial and all morphisms considered will be $k$-morphism. Let $G$ be a classical group scheme over the field $k$ and let $X$ be an fs log scheme  over the field $k$. The group scheme $G \times^{\text{fs}}_k X$  induces an a sheaf of abelian of groups on the $X_{\text{fl}}^{\text{log}}$ site.

Now, let $S$ be a sheaf of sets on the $X_{\text{fl}}^{\text{log}}$ site. We call $S$ a $G$-sheaf if it is endowed with a action of $G$.

Now, let $S_1$ be a $G_1$ sheaf and $S_2$ be a $G_2$ sheaf where $G_1$ and $G_2$ are group schemes over $k$ and $S_1$ and $S_2$ are sheaves of sets on $X_{\text{fl}}^{\text{log}}$ site. Let $h$ be a group scheme homomorphism from $G_1$ to $G_2$. We say a morphism of sheaf $f$ from $S_1$ to $S_2$ is a $h$-morphism if for every $g$ in $G_1(U)$ and $s$ in $S_1(U)$, $f_U(gs)=h_U(g)f_U(s)$.

A $G$-sheaf $S$ is a called a \textit{trivial torsor} if it is isomorphic to $G$ as a $G$-sheaf.

\begin{defn}
Let $S$ be a $G$-sheaf on $X_{\text{fl}}^{\text{log}}$ site where $G$ is a group scheme over $k$. We say $S$ is a $G$-torsor if there exists a covering $\{U_i \rightarrow X\}_{i \in I} $ in $X_{\text{fl}}^{\text{log}}$ such that $S_{U_i}$ is isomorphic to $G\times_k U_i$ as a $G\times_k U_i$-sheaf or in other words $S_{U_i}$ is a trivial $G\times_k U_i$-torsor. ( Here by $S_{U_i}$ we mean the sheaf $S$ restricted to $(U_i)_{\text{fl}}^{\text{log}}$). In other words, the sheaf $S_{U_i}$ on $(U_i)_{\text{fl}}^{\text{log}}$ site  is representable by the log scheme $G\times_k U_i$.

\end{defn}

\begin{rem}
It is not known if torsors over finite group schemes om Kummer log flat topology are representable. Therefore, we will be dealing with sheaf torsors.
\end{rem}

\section{Log Nori fundamental group scheme}\label{sec6}

Let $X$ be an fs log scheme log smooth over a field $k$ and $x \in X(k)$ . Let $\mathcal{FT}(X,x)$ denote the following category: the objects of this category are triples of the form $(S,G,\alpha)$ where $S$ is a $G$-torsor on $X_{fl}^{log}$ site and $G$ is a classical finite group scheme over the log point $(\spec(k), e)$ with strict structure morphism and $\alpha$ is a trivialization $S|_{x} \cong G$ and morphisms between $(S_1, G_1,\alpha_1)$ and $(S_2, G_2, \alpha_2)$ are of the form $(f,g )$ where $g$ is group scheme homomorphism from $G_1$ to $G_2$ and $g$ is $f$-morphism which respects the trivialization.

Now, we can define the \textit{log Nori fundamental group scheme} of $X$ as follows

$$ \pi^{N}_{log}(X) = \lim_{(S,G,\alpha) \in Ob(\mathcal{FT}(X,x))} G.$$

In general, this limit may not exists. But it does exist if the index over which the limit is taken is small and cofiltered.

To prove that the category $\mathcal{FT}(X,x)$ is cofiltered, we will show that the category $\mathcal{FT}(X)$ has fibre products. That is we will show that, if $T_1 = (S_1,G_1,t_1)$, $T_2 = (S_2,G_2,t_2)$ and $T = (S,G,t)$ are objects of $\mathcal{FT}(X)$ and $(f_1,g_1)$ is a morphism from $T_1$ to $T$ and  $(f_2,g_2)$ is a morphism from $T_2$ to $T$, then $T_1 \times_T T_2$ is $(S_1\times_S S_2, G_1 \times_G G_2, t_1\times_t t_2)$ is in $\mathcal{FT}(X,x)$.

\begin{thm}
\label{torsorsfibreproduct}
$S_1\times_S S_2$ is a $G_1 \times_G G_2$ torsor . In other words,  $(S_1\times_S S_2, G_1 \times_G G_2, t_1\times_t t_2)$ is in $\mathcal{FT}(X,x)$.
\end{thm}

\begin{proof}
Let $Y \to X$ be a Kummer log flat covering such that the torsors $S_1$, $S_2$ and $S$ are trivial $G_1$, $G_2$ and $G$-torsors respectively.
 Now using Proposition \ref{redcov} we get $X_n \to X$ and $T \to Y$ such that the $T \to X_n$ is classically flat. 
 So, if we restrict torsors $S_1, S_2$ and $S$ to $X_n$ they are torsors in the fppf topology of $X_n$ as the torsors are trivial when restricted to $T$ and $T$ is fppf cover of $X_n$. Moreover since $X_n$ is reduced this implies $S_1 \times_S S_2$ is a classical fppf $G_1 \times_G G_2$-torsor on $X_n$ by Theorem \ref{thm::nori}. Therefore, $S_1 
 \times_S S_2$ when restricted to $T$ is a classical fppf torsor and the theorem follows.
\end{proof}

The Theorem \ref{mainthm1} follows from Theorem \ref{torsorsfibreproduct}.

\section{Representability of torsors over solvable group schemes}\label{sec7}
\begin{defn}
We say a finite group scheme is solvable if it admits a subnormal series

$$ G =G_0 \supset G_1 \ldots \supset G_n ={1} $$

such that the quotients $G_i/ G_{i+1}$ are abelian.

\end{defn}

In this section we prove the following result: if $T$ is a $G$-log-flat torsor over a fs log scheme $X$ where $G$ is 
a solvable finite flat  group scheme, then $T$ is representable by an fs log scheme over $X$. We use the following 
result proved in \cite{kato2019logarithmic} to prove the above claim

\begin{prop}
Let $X$ be an fs log scheme which is locally Noetherian as a scheme. Let
$G$ be a finite flat commutative group scheme over the underlying scheme of X, which we
endow with the inverse image of $M_X$ , and let $F$ be a $G$-log flat torsors
over $X$. Then, $F$ is representable by an fs log scheme over $X$
which is log flat of Kummer type and whose underlying scheme is finite over that of $X$.
\end{prop}

Now we generalize the above result to solvable group schemes.

\begin{prop}
Let $X=(X, \mathcal{O}_X, M_X)$ be an fs log scheme which is locally Noetherian as a scheme. Let
$G$ be a solvable finite  group scheme over the underlying scheme of $X$, which we
endow with the inverse image of $M_X$ , and let $F$ be a $G$-log flat torsor
over $X$. Then, $F$ is representable by an fs log scheme over $X$
which is log flat of Kummer type and whose underlying scheme is finite over that of $X$.
\end{prop}

\begin{proof}
Since, $G$ is solvable it has a subnormal series $ G =G_0 \supset G_1 \ldots \supset G_n ={1} $ where the quotients $G_i/ G_{i+1}$ are abelian . We prove the above statement by induction on $n$ the length of the subnormal series. If $n=0$, the result follows by the previous theorem. For the inductive case, consider the exact sequence of finite group schemes where $1 \to G_1 \to G \to G_{ab} \to 1$ where $G_{ab}$ is abelian. Log flat torsors over $G_1$ are representable by inductive hypothesis. The quotient $ G_1  \setminus F$ is a $G_2$ log flat torsor. Since $G_2$ is abelian, by the previous theorem $ G_1  \setminus F$ is representable by an fs scheme $Y$ over $X$ which is again locally noetherian. Now, $F$ is a $G_1$-log flat torsor over $Y$. So, $F$ is respresentable by an fs scheme over $Y$.
\end{proof}

\section{Nilpotent pro-p quotient of the fundamental group scheme of $\mathbb{P}^1_k$ with log at $\{0, 1, \infty \}$ } \label{sec8}
\begin{defn}
A normal series of group schemes
$$ G =G_0 \supset G_1 \ldots \supset G_n ={1} $$ 
is called a central series if $G_i$ is normal in $G$ and $G_i/G_{i+1} \subset Z(G/G_{i+1})$. We say a group scheme is nilpotent if it has a central series. 
\end{defn}

A nilpotent group scheme by definition is solvable. So, the theorem proved in the last section implies all log-flat torsors over a nilpotent group scheme are representable.

We now define the pro-p nilpotent quotient of the log Nori Fundamental group scheme. Let $X$ be an fs log-smooth scheme over the field $k$, $\operatorname{char}(k)=p$. Let $\mathcal{N}(X,x)$ denote the category where objects are of the form $(T,G,t)$ where $T$ is $G$-log flat torsor over $X$ and $G$ is a nilpotent group scheme of rank  $p^r$ ($r \in \mathbb{N}$)  and $t$ is a trivialization $t : T|_{x} \cong G$.

\begin{defn}\label{nilpotent}
We define the pro-p nilpotent quotient of the log Nori fundamental group as follows
$$ (\pi^{N}_{log}(X,x))_{p-nilp} = \lim_{(T_i,G_i,t_i) \in Ob(\mathcal{N}(X,x))} G_i.$$
Here limit is the projective limit of group schemes\end{defn}

So, the nipotent quotient of the fundamental group scheme classifies torsors of nilpotent group schemes of rank power of $p$.

Let $k$ be an algebrically closed field of characteristic $p > 0$. Let $X= \mathbb{P}^{1}_k$ and $D = \{0, 1 , \infty \}$ . Let $U = X-D$ and let $j: U \to X$ be the inclusion. Let the log structure on $U$ be the trivial log structure and we endow $X$ with the direct image of this log structure. The main result proved in this section is the following.

\begin{thm}\label{thm::main2}
Let $k$ be an algebraically closed field of characteristic $p > 0$. Let $X= \mathbb{P}^{1}_k$ and $D = \{0, 1 , \infty \}$ . Let $U = X-D$ and let $j: U \to X$ be the inclusion. Let the log structure on $U$ be the trivial log structure and we endow $X$ with the direct image of this log structure. Then
$$(\pi^{N}_{log}(X,x))_{p-nilp} \cong \underset{n}{\varprojlim} \ \mu_{p^n}\times \mu_{p^n}.
$$
\end{thm}

Let $X_n = \operatorname{Proj}k[x,y,z]/(x^{p^n}+y^{p^n}-z^{p^n})$. The log structure on $X_n$, $M_{X_n}$, is defined as follows: on the affine subsets $U_1=\spec(k[x/z.y/z]/((x/z)^{p^n}+(y/z)^{p^n}-1)), U_2=\spec(k[y/x.z/x]/((y/x)^{p^n}+(z/x)^{p^n}-1))$ and $U_3=\spec(k[x/y.z/y]/((x/y)^{p^n}+(z/y)^{p^n}-1))$ the log structure is $<\mathcal{O}^{\times}_{X_n},x/z,y/z>, <\mathcal{O}^{\times}_{X_n},x/y,z/y>$ and $<\mathcal{O}^{\times}_{X_n},y/x,z/x>$ respectively. Consider the morphism $\phi_n:X_n \to X$ where $[x:y:z]$ goes to $[x^{p^n}:y^{p^n}:z^{p^n}]$ ( the underlying scheme of $X$ is $\operatorname{Proj}k[x,y,z]/(x+y-z)=  \mathbb{P}^{1}_k $ ). Now, $\phi_n:X_n \to X$ is a $\mu_{p^n} \times \mu_{p^n}$-log flat torsor as proved below.

\begin{prop}
The fs log scheme $X_n$ is a $\mu_{p^n} \times \mu_{p^n}$-log flat torsor over $X$.
\end{prop}

\begin{proof}
It is enough to show the $X_n$ is $\mu_{p^n} \times \mu_{p^n}$-log flat torsor over $X$ when restricted to some open affine cover of $X$.
Consider the open affine cover $\mathbb{A}_k^1 \to X$ where $\mathbb{A}_k^1$ has log at $0$ and $1$. Consider $X_n^{'} = X_n \times^{\text{fs}}_{X} \mathbb{A}_k^1$. It is enough to show that $X_n^{'}$ is $\mu_{p^n} \times \mu_{p^n}$-log flat torsor over  $\mathbb{A}_k^1$. We denote by $\mathbb{N} \to \mathbb{N}^{1/p^n}$ the map $\mathbb{N} \to \mathbb{N}$ where $1$ goes to $p^n$.To prove  $X_n^{'}$ is $\mu_{p^n} \times \mu_{p^n}$-log flat torsor over $\mathbb{A}_k^1$ we need to show that $X_n^{'} \times^{\text{fs}}_{\mathbb{A}_k^1} X_n^{'} \cong X_n^{'} \times (\mu_{p^n} \times \mu_{p^n}) $.  Note that $X_n^{'} = \mathbb{A}_k^1 \times^{\text{fs}}_ {\mathbb{Z}[\mathbb{N},\mathbb{N}]} \mathbb{Z}[\mathbb{N}^{1/p^n},\mathbb{N}^{1/p^n}]$. First we compute the fibred product $X_n^{'} \times_{\mathbb{A}_k^1} X_n^{'}$ in the category of fine log schemes. Then $X_n^{'} \times_{\mathbb{A}_k^1} X_n^{'} = (\mathbb{A}_k^1 \times_ {\mathbb{Z}[\mathbb{N},\mathbb{N}]} \mathbb{Z}[\mathbb{N}^{1/p^n},\mathbb{N}^{1/p^n}]) \times_{\mathbb{A}_k^1   } (\mathbb{A}_k^1 \times_ {\mathbb{Z}[\mathbb{N},\mathbb{N}]} \mathbb{Z}[\mathbb{N}^{1/p^n},\mathbb{N}^{1/p^n}]) = \mathbb{A}_k^1 \times_ {\mathbb{Z}[\mathbb{N},\mathbb{N}]} (\mathbb{Z}[\mathbb{N}^{1/p^n},\mathbb{N}^{1/p^n}]) \times_ {\mathbb{Z}[\mathbb{N},\mathbb{N}]} \mathbb{Z}[\mathbb{N}^{1/p^n},\mathbb{N}^{1/p^n}]) = \mathbb{A}_k^1 \times_ {\mathbb{Z}[\mathbb{N},\mathbb{N}]}  \mathbb{Z}[P]$ where $P = (\mathbb{N}^{1/p^n}\oplus \mathbb{N}^{1/p^n}) \oplus_{\mathbb{N} \oplus \mathbb{N}} (\mathbb{N}^{1/p^n}\oplus \mathbb{N}^{1/p^n})$.Note that the $(P^{\text{int}})^{\text{sat}} =  (\mathbb{N}^{1/p^n} \oplus \mathbb{Z}/{p^n \mathbb{Z}})\oplus (\mathbb{N}^{1/p^n} \oplus \mathbb{Z}/{p^n \mathbb{Z}})$. So, we finally get $X_n^{'} \times^{\text{fs}}_{\mathbb{A}_k^1} X_n^{'} =  (\mathbb{A}_k^1 \times_ {\mathbb{Z}[\mathbb{N},\mathbb{N}]}  \mathbb{Z}[P]) \times_{\mathbb{Z}[P]   } \mathbb{Z}[(P^{\text{int}})^{\text{sat}}] = X_n^{'} \times (\mu_{p^n} \times \mu_{p^n})$. Similar it can be shown for other open affine cover $X-\{1, \infty\}$ and $X-\{0, \infty\}$.
\end{proof}

Also, there's map $h_n:X_{n+1} \to X_{n}$ such that $f_n:([x:y:z]) = [x^p:y^p:z^p]$. Also let $g_n$ denote the group scheme homomorphism $\mu_{p^{n+1}}\times\mu_{p^{n+1}} \to \mu_{p^{n}}\times\mu_{p^{n}}$ such $(x,y) \in \mu_{p^{n+1}}(R)\times\mu_{p^{n+1}}(R)$ goes to $(x^p,y^p) \in \mu_{p^{n}}(R)\times  \mu_{p^{n}}(R)$. Since, the group scheme $\mu_{p^{n}}\times \mu_{p^{n}}$ is abelian therefore nilpotent. So,  from the projective system of log flat torsors $(X_n,\mu_{p^n} \times \mu_{p^n},t_n)$ we get the following map 

\begin{equation}
\label{mainmap}
    f:(\pi^{N}_{log}(X,x))_{p-nilp} \to \underset{n}{\varprojlim} \ \mu_{p^n}\times \mu_{p^n}
\end{equation}

where the maps for the projective limit is given by $g_n$. 

In the following two sections, we will prove that $f$ is an isomorphism.

\subsection{ Proof of surjectivity of $f$} \label{sec8parta}
\leavevmode

First, we show that $f$ is surjective. For this it enough to show that the canonical map $f_n$, $\pi^{N}_{log}(X,x))_{p-nilp} \to  \mu_{p^n}\times \mu_{p^n}$ is surjective for every $n$. 

\begin{lem}
Suppose the canonical map from $f_n: \pi^{N}_{log}(X,x))_{p-nilp} \to  \mu_{p^n}\times \mu_{p^n}$ is not surjective. The there exists a surjective group scheme morphism $s$ from $\mu_{p^n}\times \mu_{p^n}$ to $\mu_p$ such that $s \circ f_n$ is zero morphism.
\end{lem}
\begin{proof}
Let $\operatorname{coker}(f_n):\mu_{p^n}\times \mu_{p^n} \to  \operatorname{coker}(f_n)$ be the cokernel of the morphism $f_n$. Since $f_n$ is not surjective, $\operatorname{coker}(f_n)$ is not the trivial group scheme.The $\operatorname{coker}(f_n)$ is a non-trivial finite group scheme whose rank is a power of $p$. So, it must have a simple finite quotient which is $\alpha_p$, $\mu_p$ or $\mathbb{Z}/p\mathbb{Z}$. But since any quotient of $\operatorname{coker}(f_n)$ is also a quotient of $\mu_{p^n}\times \mu_{p^n}$ so the quotient should be $\mu_p$. Therefore, there exists a surjective morphism $s'$ from  $\operatorname{coker}(f_n)$ to $\mu_p$. Then we can take $s$ to be $s' \circ  \operatorname{coker}(f_n)$. We can see that $s \circ f_n$ is the zero morphism.
\end{proof}

To prove that $f_n$ is surjective it is enough to show that for any surjective morphism $s$ from $ \mu_{p^n}\times \mu_{p^n}$ to $\mu_{p}$ the image $s_{\text{in}}(\alpha) \neq 0$ where $s_{\text{in}}:H^1(X, \mu_{p^n}\times \mu_{p^n} ) \to H^1(X, \mu_{p}) $ is the induced map on cohomology groups and $\alpha$ corresponds to the $ \mu_{p^n}\times \mu_{p^n}$-torsor $X_n$. We have the following exact sequences.

\begin{tikzcd}
0 \arrow[r] & \mu_{p^n}\times \mu_{p^n}  \arrow[r] \arrow[d,"s_{i,j}"] & M^{\text{gp}} \times M^{\text{gp}} \arrow[r,"p^n \times p^n"] \arrow[d, "t_{i,j}"] & M^{\text{gp}} \times M^{\text{gp}}   \arrow[d,"p_{i,j}"] \arrow[r] &0\\
0 \arrow[r]  &\mu_{p} \arrow[r]  & M^{\text{gp}}  \arrow[r, "p"]  &M^{\text{gp}} \arrow[r] &0
\end{tikzcd}

Here $s_{i,j}$ is the morphism where $(a,b) \in \mu_{p^n}(R) \times \mu_{p^n}(R)$ goes $a^{ip^{n-1}}b^{jp^{n-1}} \in \mu_p(R) $ and $0 \leq i,j \leq p-1$ and $(i,j) \neq (0,0)$. Note that every surjective morphism from $ \mu_{p^n}\times \mu_{p^n}$ to $\mu_{p}$ is of the form $s_{i,j}$. The map $t_{i,j}$ takes $(a,b) \in M^{\text{gp}}(U) \times M^{\text{gp}}(U) $ to  $a^{ip^{n-1}}b^{jp^{n-1}}$. The map $p_{i,j}$ takes  $(a,b) \in  M^{\text{gp}}(R) \times M^{\text{gp}}(R)$ to $a^ib^j \in M^{\text{gp}}(R)$.

Now consider the corresponding long exact sequence.

\begin{tikzcd}[column sep=tiny]
0 \arrow[r] &H^0_{logfl}(X, \mu_{p^n}\times \mu_{p^n} )  \arrow[r] \arrow[d,"s_{i,j}"] & H^0_{logfl}(X, M^{\text{gp}} \times M^{\text{gp}}) \arrow[r,"p^n \times p^n"] \arrow[d, "t_{i,j}"] & H^0_{logfl}(X, M^{\text{gp}} \times M^{\text{gp}}) \arrow[d,"p_{i,j}"] \arrow[r]  &H^1_{logfl}(X,\mu_{p^n}\times \mu_{p^n})  \arrow[d,"s_{i,j}"]\\
0 \arrow[r]  &H^0_{logfl}(X,\mu_{p^n}) \arrow[r]  & H^0_{logfl}(X, M^{\text{gp}})  \arrow[r]  &H^0_{logfl}(X,M^{\text{gp}}) \arrow[r] &H^1_{logfl}(X,\mu_{p^m}) 
\end{tikzcd}

Now $(t, t-1) \in  H^0_{logfl}(X, M^{\text{gp}} \times M^{\text{gp}})$ corresponds to the $ \mu_{p^n}\times \mu_{p^n}$-torsor $X_n$. Under the map $p_{i,j}$ ,  $(t,t-1)$ goes to $t^i(t-1)^j$ which is not a $p$-th power. Therefore, $s_{i,j}(\alpha) \neq 0 $ where $\alpha \in H^1(X,\mu_{p^n}\times \mu_{p^n})$ corresponds to the $ \mu_{p^n}\times \mu_{p^n}$-torsor $X_n$. This proves that $f_n$ is surjective.

\subsection{ Proof of injectivity of $f$} \label{sec8partb}
\leavevmode
Now we show that $f$ is injective.
Let $R$ be a $k$-algebra and $X$ be an $R$-scheme.

Consider the exact sequence of group schemes over $R$

\begin{equation}
    1 \to N \to G \to H \to 1
\end{equation}

Let $U$ be a $G$-torsor over $X$. Consider the torsor $U/N$ and assume that it is representable by the log scheme $Y$. Then $U \to Y$ is a $N$-torsor over $Y$ and $Y \to X$ is a $H$-torsor over $X$.

Let $\sigma \in H(R)$ then it induces a morphism $f_\sigma: Y \to Y$. Suppose $P = N/Q$ is an abelian quotient of $N$ where $Q$ is a normal subgroup scheme of $N$ and $G$. Then $U/Q= Z$ is a $P$-torsor over $Y$.

Now $\sigma$ acts on the group scheme $P$ by inner-automorphisms. We denote this action by $i_{\sigma}: P \to P$. Now consider the base change of the torsor $Z$ by the morphism 
$f_{\sigma}:Y \to Y$. Then we get a new $P$-torsor $Z'$ over $Y$ and $Z'= Z \times^{P} P$, here we are treating $P$ as $P$-torsor by the action 
$i_{\sigma}$. In other words, let $f_{\sigma}: H^1(Y,P) \to H^1(Y,P)$ be 
the map induced on the cohomology groups by the morphism $f_\sigma: Y \to Y$. Let $\alpha \in H^1(Y,P)$ correspond to the $P$-torsor $Z$. Then 
$f_{\sigma}( \alpha)$ corresponds to the torsor  $Z \times^{P} P$.

Now consider the map
\begin{equation}
   f:(\pi^{N}_{log}(X,x))_{p-nilp} \to \underset{n}{\varprojlim} \ \mu_{p^n}\times \mu_{p^n}.
\end{equation}

Suppose the kernel of $f$, $ \operatorname{ker}(f)$, is non-trivial. 
Then it should have a quotient $H$ where $ H=\alpha_p, \mu_p$ or $\mathbb{Z}/p \mathbb{Z}$.
Now consider the map $\phi_n: \pi^{N}_{log}(X,x))_{p-nilp} \to \mu_{p^n} \times \mu_{p^n}$. If $\phi_n$ has quotient $ H=\alpha_p$ or $\mathbb{Z}/p \mathbb{Z}$ then acts $\mu_{p^n} \times \mu_{p^n}$ acts on $H$ by inner-automorphisms and since $\pi^{N}_{log}(X,x))_{p-nilp}$ is nilpotent  these inner-automorphisms are trivial. 
Now consider the action of 
$\mu_{p^n}\times \mu_{p^n} (R)$ on $X_n \otimes R = \operatorname{Proj}(R[T_1, T_2, 
T_3]/(T_1^{p^n}+T_2^{p^n}-T_3^{p^n}))$ where $(a,b) \in \mu_{p^n}\times \mu_{p^n} (R)$ induces the map $h_{(a,b)}: X_n \otimes R \rightarrow X_n \otimes R$  where $h_{(a,b)}(T_1) = aT_1$ and $h_{(a,b)}(T_2) = bT_2$. The map $h_{(a,b)}$ also induces a corresponding map on $ H^1_{logfl}(X_n \otimes R, H)$. Then there should 
exist $ \theta \neq 0 \in H^1_{logfl}(X_n, H) \subset H^1_{logfl}(X_n \otimes R, H)$ such that $h_{(a,b)}( \theta) 
=\theta$ for all $(a,b) \in \mu_{p^n} \times \mu_{p^n}(R)$ and for all $k$-algebras $R$.

Let $X$ be a locally noetherian fs log scheme.
Let $ \epsilon: X^{\log}_{fl} \to X^{cl}_{fl}$ be a morphism of site that is induced by the functor that takes a classical scheme $U$ in $X^{cl}_{fl}$ and endow it with the inverse image of the log structure of $X$ in $X^{\log}_{fl}$. Now, let $G$ be a finite commutative group scheme of $X$ where the log structure on $G$ is the inverse image log structure on $X$. Now, we can apply the Leray spectral sequence

\begin{equation} \label{leray}
0 \to H_{fl}^1(X, G ) \to H^1_{logfl}(X, G) \to H_{fl}^0(X, R^{1}\epsilon_{*}(G) ) \to H_{fl}^2(X,G) \to H_{logfl}^2(X^, G) 
\end{equation}

To use the above exact sequence we need to compute the first higher direct image of $G$,$R^{1}\epsilon_{*}(G)$. 

For this we need this following theorem proved in Theorem 4.1 of \cite{kato2019logarithmic} .

\begin{thm}\label{thmr}
Let $X$ be an fs log scheme and assume that $X$ is locally Noetherian as a
scheme.  Let $G$ be a commutative
group scheme over the underlying scheme of $X$ satisfying either one of the following two conditions. \\
(i) $G$ is finite flat over the underlying scheme of $X$. \\
(ii) $G$ is smooth and affine over the underlying scheme of $X$. \\
 Then we have a canonical isomorphism
$$R^{1} \epsilon _{*} G 
= \underset{n \neq 0}{\varinjlim} {\Hom}(\mu_{n}, G) \otimes_{\mathbb{Z}} (\mathbb{G}_{m,\log} /\mathbb{G}_m ),
$$
where $n$ ranges over all non-zero integers and the inductive limit is taken with respect to
the canonical projections $\mu_{mn} \to \mu_{n}$.
Here $\mathbb{G}_{m}$ is the functor $T \to \Gamma(T, \mathcal{O}_{T}^ {\times} )$ on (fs/X) and $\mu_n = \operatorname{Ker} (n : \mathbb{G}_m \to \mathbb{G}_m )$
$(n \neq 0)$. The quotient $\mathbb{G}_{m,log} /\mathbb{G}_m$ here is taken in the categories of sheaves on $X^{cl}_{fl}$.
\end{thm}

To compute $H_{log_fl}^1(X, G)$ when $G = \alpha_p$, $\mathbb{Z}/p\mathbb{Z}$ we need the following simple lemma. Here $p$ is the characteristic of the field $k$.

\begin{lem}
$R^{1}\epsilon_{*}(G) = 0$ if $G = \alpha_p, \mathbb{Z}/p\mathbb{Z}$.
\end{lem}
\begin{proof}
From \ref{thmr} we get $R^{1} \epsilon _{*} G 
= \underset{n \neq 0}{\varinjlim} {\Hom}(\mu_{n}, G) \otimes_{\mathbb{Z}} (G_{m,\log} /G_m )$. Since $\operatorname{Hom}(\mu_{n},G)=0$  when $G = \alpha_p, \mathbb{Z}/p\mathbb{Z}$ , we get $\underset{n \neq 0}{\varinjlim} {\Hom}(\mu_{n}, G)= 0$ and the result follows.

\end{proof}

Now using \ref{leray} and the above lemma, we get 
$ H^1_{fl}(X, G ) \cong H^1_{log_fl}(X, G)$ when $G = \alpha_p, \mathbb{Z}/p\mathbb{Z}$.

To compute $H^1_{fl}(X_n, G )$ we use the following two exact sequence in the fppf topology

\begin{equation}\label{alpha} 0 \to \alpha_p \to \mathbb{G}_a \xrightarrow[]{F} \mathbb{G}_a \to 0 
\end{equation}

\begin{equation}\label{artin} 0 \to \mathbb{Z}/p\mathbb{Z} \to  \mathbb{G}_a \xrightarrow[]{F -id} \mathbb{G}_a \to 0
\end{equation}

Using \ref{alpha} we get the following long exact sequence
\begin{equation}
0 \to \Gamma(X_n, \alpha_p) \to \Gamma(X_n,\mathbb{G}_a) \xrightarrow[]{F} \Gamma(X_n,\mathbb{G}_a) \to H^1_{fl}( X_n, \alpha_p) \to H^1_{fl}( X, G_a) \to H^1_{fl}( X, G_a)
\end{equation}

\begin{equation}
0 \to 0 \to k \xrightarrow[]{F} k \to H^1_{fl}( X, \alpha_p) \to H^1_{fl}( X_n, G_a) \xrightarrow[]{F} H^1( X_n, G_a)
\end{equation}

Since, $k$ is algebraically closed the map $k \xrightarrow[]{F} k$ is surjective, therefore the map  $H^1( X_n, \alpha_p) \to H^1_{fl}( X_n, G_a))$ is injective. Also note that $H^1_{fl}( X_n, G_a) = H^1_{zar}(X_n, G_a )= H^2_{zar}(\mathbb{P}^2,\mathcal{O}(-p^n) )= (\frac{1}{T_0T_1T_2}k[\frac{1}{T_0},\frac{1}{T_1}, \frac{1}{T_2}])_{(-p^n)}$. 

We now prove the following proposition

\begin{prop}
\label{alphafixed}
Let $R$ be a $k$-algebra then we have we have a map $i :H^1_{fl}(X_n, \alpha_p) \to H^1_{fl}(X_n \otimes_k R, \alpha_p)$. Also for every $(a,b) \in \mu_{p^n}(R) \times \mu_{p^n}(R)$, $h_{(a,b)}$ induces map $h_{(a,b)}:  H^1_{fl}(X_n \otimes_k R, \alpha_p) \to H^1_{fl}(X_n \otimes_k R, \alpha_p) $.

Let $H^1_{fl}(X_n, \alpha_p)^{\mu_{p^n}(R) \times \mu_{p^n}(R)} = \{ x \in H^1(X_n, \alpha_p):\text{for all } (a,b) \in \mu_{p^n}(R) \times \mu_{p^n}(R)\text{ , } h_{(a,b)}(i(x)) =i(x) \}$. There exists a $k$-algebra $R$ such that 
\begin{equation*}
H^1_{fl}(X_n, \alpha_p)^{\mu_{p^n}(R) \times \mu_{p^n}(R)} = 0
\end{equation*}
\end{prop}
\begin{proof}
We need to show that there exists a $k$-algebra $R$ and  $(a,b) \in  \mu_{p^n}(R) \times \mu_{p^n}(R)$ such that $h_{(a,b)}(i(x)) \neq i(x)$ for all $x \in H^1(X_n, \alpha_p), x \neq 0$.

We start the with following exact sequences

$ 0 \to H^1_{fl}(X_n, \alpha_p) \to H^1_{fl}(X_n, \mathbb{G}_a)$,

Also there is the map $i:H^1(X_n, \mathbb{G}_a) \to H^1(X_n \otimes_k R, \mathbb{G}_a)$. Since, $\mathbb{G}_a$ is a quasi-coherent sheaf we have $H^1_{fl}(X_n \otimes_k R, \mathbb{G}_a) = H^1_{fl}(X_n , \mathbb{G}_a) \otimes_k R$ and $i$ is injective. Then we get the following commutative diagram.

\begin{tikzcd}
H^1_{fl}(X_n, \alpha_p) \arrow[r,"p"] \arrow[d, "i"] & H^1(X_n, \mathbb{G}_a) \arrow[d, "i"]\\
H^1_{fl}(X_n \otimes_k R, \alpha_p) \arrow[r, ,"p"] \arrow[d, "h_{(a,b)}"]  & H^1_{fl}(X_n \otimes_k R, \mathbb{G}_a) \arrow[d, "h_{(a,b)}"] \\
H^1_{fl}(X_n \otimes_k R, \alpha_p) \arrow[r,,"p"]  & H^1_{fl}(X_n \otimes_k R, \mathbb{G}_a)  
\end{tikzcd}

Let $x \in H_{fl}^1(X_n, \alpha_p), x \neq 0$, then if $h_{(a,b)}(i((p(x))) \neq i(p(x))$ then it is clear that $h_{(a,b)}(i(x)) \neq h_{(a,b)}(i(x))$. So, if we have to show for all $x \in H^1_{fl}(X_n \otimes_k R, \mathbb{G}_a), x \neq 0$, $h_{(a,b)}(x) \neq x$ for some $(a,b) \in \mu_{p^n}(R) \times \mu_{p^n}(R)$.

Now, $H^1_{fl}(X_n \otimes_k R, \mathbb{G}_a) = \left (\frac{1}{T_0T_1T_2}R[\frac{1}{T_0},\frac{1}{T_1}, \frac{1}{T_2}] \right )_{(-p^n)}$ and the map $h_{(a,b)}$ takes any monomial $\frac{1}{T_0^xT_1^yT_2^z}$ to $\frac{1}{a^xT_0^xb^yT_1^yT_2^z}$.

Now, we can take our $R$ to be $k[t]/(t^{p^n}-1)$ and $(a,b)= (t,1) \in \mu_{p^n} \times \mu_{p^n}(R)$ and we get our result.
\end{proof}

Now since $H^1_{fl}(X_n, \alpha_p) = H^1_{log fl}(X_n, \alpha_p)$ we get the following proposition

\begin{prop}
\label{alphafixedlog}
Let $R$ be a $k$-algebra then we have we have a map $i :H^1_{logfl}(X_n, \alpha_p) \to H^1_{logfl}(X_n \otimes_k R, \alpha_p)$. Also for every $(a,b) \in \mu_{p^n}(R) \times \mu_{p^n}(R)$, $h_{(a,b)}$ induces map $h_{(a,b)}:  H^1_{logfl}(X_n \otimes_k R, \alpha_p) \to H^1_{logfl}(X_n \otimes_k R, \alpha_p) $.

Let $H^1_{logfl}(X_n, \alpha_p)^{\mu_{p^n}(R) \times \mu_{p^n}(R)} = \{ x \in H^1_{logfl }(X_n, \alpha_p):\text{for all } (a,b) \in \mu_{p^n}(R) \times \mu_{p^n}(R)\text{ , } h_{(a,b)}(i(x)) =i(x) \}$. There exists a $k$-algebra $R$ such that 
\begin{equation*}
H^1_{logfl}(X_n, \alpha_p)^{\mu_{p^n}(R) \times \mu_{p^n}(R)} = 0
\end{equation*}
\end{prop}

We also have an equivalent proposition for $\mathbb{Z}/p\mathbb{Z}$.

\begin{prop}
\label{zpzfixedlog}
Let $R$ be a $k$-algebra then we have we have a map $i :H^1_{logfl}(X_n, \mathbb{Z}/p\mathbb{Z}) \to H^1_{logfl}(X_n \otimes_k R, \mathbb{Z}/p\mathbb{Z})$. Also for every $(a,b) \in \mu_{p^n}(R) \times \mu_{p^n}(R)$, $h_{(a,b)}$ induces map $h_{(a,b)}:  H^1_{logfl}(X_n \otimes_k R, \mathbb{Z}/p\mathbb{Z}) \to H^1_{logfl}(X_n \otimes_k R, \mathbb{Z}/p\mathbb{Z}) $.

Let $H^1_{logfl}(X_n, \mathbb{Z}/p\mathbb{Z})^{\mu_{p^n}(R) \times \mu_{p^n}(R)} = \{ x \in H^1_{logfl }(X_n, \mathbb{Z}/p\mathbb{Z}):\text{for all } (a,b) \in \mu_{p^n}(R) \times \mu_{p^n}(R)\text{ , } h_{(a,b)}(i(x)) =i(x) \}$. There exists a $k$-algebra $R$ such that 
\begin{equation*}
H^1_{logfl}(X_n, \mathbb{Z}/p\mathbb{Z})^{\mu_{p^n}(R) \times \mu_{p^n}(R)} = 0
\end{equation*}
\end{prop}
\begin{proof}
Similar to Proposition \ref{alphafixedlog}.
\end{proof}

Now we will prove a proposition similar to Proposition \ref{alphafixedlog} but for $H^1_{logfl}(X_n,\mu_p)$. To prove this we will use induction. First note that the there exists a closed embedding
$j: X_{n-1} \otimes_{k} R \to X_{n} \otimes_{k} R$ which is induced by the following morphism of graded rings $\phi:R[T_0,T_1,T_2]/(T_0^{p^{n}}+T_1^{p^{n}}-T_2^{p^{n}}) \to R[T_0,T_1,T_2]/(T_0^{p^{n-1}}+T_1^{p^{n-1}}-T_2^{p^{n-1}})$ where $\phi(T_0) = T_0,\phi(T_1) =T_1, \phi(T_2) = T_2$.

We get the following exact sequence

\begin{equation}
\label{closedembedding}
0 \to I \to \mathcal{O}_ {X_{n} \otimes_{k} R} \to j_{*} \mathcal{O}_ {X_{n-1} \otimes_{k} R} \to 0
\end{equation}
Note that $I^p=0$.
\begin{lem}
\label{exactsequence}
The following sequence of Zariski sheaves is exact
\begin{equation}
0 \to I \to \mathcal{O}^{\times}_ {X_{n} \otimes_{k} R} \to j_{*} \mathcal{O}^{\times}_ {X_{n-1} \otimes_{k} R} \to 0
\end{equation}

Here the first map from $I(U)$ to $\mathcal{O}^{\times}_ {X_{n-1} \otimes_{k} R}(U)$ is given by the truncated exponential sequence $exp_{p-1}(x) = \sum_{i=0}^{p-1}x^i/i!$. The group structure in $I(U)$ is additive where as the group structure in $\mathcal{O}^{\times}_ {X_{n-1} \otimes_{k} R}(U)$ is multiplicative.
\end{lem}
\begin{proof}
We start with short exact sequence \ref{closedembedding} and get the following sequence

\begin{equation}
0 \to I \to \mathcal{O}^{\times}_ {X_{n} \otimes_{k} R} \to j_{*} \mathcal{O}^{\times}_ {X_{n-1} \otimes_{k} R} \to 0
\end{equation}

For left exactness, we need to check that the truncated exponential function $exp_{p-1}(x)$ induces an isomorphism between the sheaf $I$ and $1+I$. The truncated exponential function is a homomorphism since $exp_{p-1}(x+y) = exp_{p-1}(x).exp_{p-1}(y)$, which follows from the fact $I^p=0$. To prove it is an isomorphism we define the truncated logarithm, $\log_{p-1}(x) = \sum_{i=1}^{p-1}(-1)(x-1)^i/i$. One can check that $x =exp_{p-1}(\log_{p-1}(x)) = \log_{p-1}(exp_{p-1}(x)) $. This proves the left exactness. The right exactness follows from the right exactness of \ref{closedembedding}.
\end{proof}

Now we will prove the following proposition
\begin{prop}
\label{mufixed}
Let $R$ be a $k$-algebra then we have a map $i :H^1_{fl}(X_n, \mu_p) \to H^1_{fl}(X_n \otimes_k R, \mu_p)$. Also for every $(a,b) \in \mu_{p^n}(R) \times \mu_{p^n}(R)$, $h_{(a,b)}$ induces map $h_{(a,b)}:  H^1_{fl}(X_n \otimes_k R, \mu_p) \to H^1_{fl}(X_n \otimes_k R, \mu_p) $.

Let $H^1_{fl}(X_n, \mu_p)^{\mu_{p^n}(R) \times \mu_{p^n}(R)} = \{ x \in H^1_{fl}(X_n, \mu_p):\text{for all } (a,b) \in \mu_{p^n}(R) \times \mu_{p^n}(R)\text{ ,} h_{(a,b)}(i(x)) =i(x) \}$. There exists a $k$-algebra $R$ such that 
\begin{equation*}
H^1_{fl}(X_n, \mu_p)^{\mu_{p^n}(R) \times \mu_{p^n}(R)} = 0
\end{equation*}
\end{prop}

\begin{proof}
We will use induction for the proof. We need to show that there exists a $k$-algebra $R$ and $(a,b) \in \mu_{p^n}(R) \times \mu_{p^n}(R)$ such that for all $x \in H^1_{fl}(X_n,\mu_p)$ and $x \neq 0$, $h(i(x)) \neq i(x)$.

Assume by inductive hypothesis that such $h_{(a,b)}$ exists for $X_n$. We will show that it also exists for $X_{n+1}$.

Consider the following commutative diagram of schemes

\begin{tikzcd}
X_{n}\otimes_k R \arrow[r, "j"] \arrow[d, "h_{(a,b)}"] & X_{n+1}\otimes_k R \arrow[d, "h_{(a,b)}"] \\
X_{n} \otimes_k R \arrow[r, "j"]  & X_{n+1}\otimes_k R  \\
\end{tikzcd}

Now using the Lemma \ref{exactsequence} and the above commutative diagram we get following commutative diagram

\begin{tikzcd}
H^1_{fl}( X_{n+1}\otimes_k R,I) \arrow[r] \arrow[d,"h_{(a,b)}"] &H^1_{fl}( X_{n+1}\otimes_k R,\mathcal{O}^{\times}
_ {X_{n+1} \otimes_{k} R}) \arrow[r, "j_{\text{in}}"]  \arrow[d,"h_{(a,b)}"] &H^1_{fl}( X_{n+1}\otimes_k R,j_{*}\mathcal{O}
^{\times}_ {X_{n} \otimes_{k} R})  \arrow[d,"h_{(a,b)}"] \\
H^1_{fl}( X_{n+1}\otimes_k R,I) \arrow[r] &H^1_{fl}( X_{n+1}\otimes_k R,\mathcal{O}^{\times}
_ {X_{n+1} \otimes_{k} R}) \arrow[r, "j_{\text{in}}"] &H^1_{fl}( X_{n+1}\otimes_k R,j_{*}\mathcal{O}
^{\times}_ {X_{n} \otimes_{k} R})
\end{tikzcd}

Now we add the map $i:H^1(X_n, \mathcal{O}^{\times}_{X_n}) \to H^1(X_n \otimes_{k} R, \mathcal{O}^{\times}_{X_n})$ and $i:H^1(X_n, I) \to H^1(X_n \otimes_{k} R, I)$ and we get the following commutative diagram

\begin{tikzcd}
H^1_{fl}( X_{n+1},I) \arrow[r] \arrow[d,"i"] &H^1_{fl}( X_{n+1},\mathcal{O}^{\times}
_ {X_{n+1}}) \arrow[r, "j_{\text{in}}"]  \arrow[d,"i"] &H^1_{fl}( X_{n+1},j_{*}\mathcal{O}
^{\times}_ {X_{n}})  \arrow[d,"i"] \\
H^1_{fl}( X_{n+1}\otimes_k R,I) \arrow[r] \arrow[d,"h_{(a,b)}"] &H^1_{fl}( X_{n+1}\otimes_k R,\mathcal{O}^{\times}
_ {X_{n+1} \otimes_{k} R}) \arrow[r, "j_{\text{in}}"]  \arrow[d,"h_{(a,b)}"] &H^1_{fl}( X_{n+1}\otimes_k R,j_{*}\mathcal{O}
^{\times}_ {X_{n} \otimes_{k} R})  \arrow[d,"h_{(a,b)}"] \\
H^1_{fl}( X_{n+1}\otimes_k R,I) \arrow[r] &H^1_{fl}( X_{n+1}\otimes_k R,\mathcal{O}^{\times}
_ {X_{n+1} \otimes_{k} R}) \arrow[r, "j_{\text{in}}"] &H^1_{fl}( X_{n+1}\otimes_k R,j_{*}\mathcal{O}
^{\times}_ {X_{n} \otimes_{k} R})
\end{tikzcd}

Let $x$ be a $p$-torsion element of $H^1_{fl}( X_{n+1},\mathcal{O}^{\times}_{X_{n+1}})$. Then $j_{\text{in}}(x)=x'$ is a $p$-torsion element of $H^1_{fl}( X_{n+1},j_{*}\mathcal{O}^{\times}_{X_{n}})= H^1_{fl}( X_{n},\mathcal{O}^{\times}_{X_{n}}) $. If $j_{\text{in}}(x) \neq 0$ by inductive hypothesis we get $ h_{(a,b)}(i(x')) \neq i(x')$ but this implies $ h_{(a,b)}(i(x)) \neq i(x)$ . Now consider the case when $j_{\text{in}}(x)=0$. Then $ x \in H^1_{fl}( X_{n+1},I)$. Now we have the following exact sequence \ref{closedembedding} we get the following exact sequence

\begin{equation}
0 \to H^1_{fl}( X_{n+1}\otimes_k R,I) \to H^1_{fl}( X_{n+1}\otimes_k R,\mathcal{O}
_ {X_{n+1} \otimes_{k} R}) \to H^1_{fl}( X_{n+1}\otimes_k R,j_{*}\mathcal{O}
_ {X_{n} \otimes_{k} R}) \to 0
\end{equation}

Now combining with the action of $h_{(a,b)}$

We get the following commutative diagram

\begin{tikzcd}
H^1_{fl}( X_{n+1},I) \arrow[r] \arrow[d, "i"] &H^1_{fl}( X_{n+1},\mathcal{O}
_ {X_{n+1}}) \arrow[r] \arrow[d, "i"]  &H^1_{fl}( X_{n+1},j_{*}\mathcal{O}
_ {X_{n}}) \arrow[d, "i"] \\
H^1_{fl}( X_{n+1}\otimes_k R,I) \arrow[r] \arrow[d, "h_{(a,b)}"] &H^1_{fl}( X_{n+1}\otimes_k R,\mathcal{O}
_ {X_{n+1} \otimes_{k} R}) \arrow[r] \arrow[d, "h_{(a,b)}"]  &H^1_{fl}( X_{n+1}\otimes_k R,j_{*}\mathcal{O}
_ {X_{n} \otimes_{k} R}) \arrow[d, "h_{(a,b)}"]\\
H^1_{fl}( X_{n+1}\otimes_k R,I) \arrow[r]  &H^1_{fl}( X_{n+1}\otimes_k R,\mathcal{O}
_ {X_{n+1} \otimes_{k} R}) \arrow[r]  &H^1_{fl}( X_{n+1}\otimes_k R,j_{*}\mathcal{O}
_ {X_{n} \otimes_{k} R})\\
\end{tikzcd}

But we have already seen in the proof of Proposition \ref{alphafixed} that we can find $h_{(a,b)}$ such that for any $x \in H^1_{fl}( X_{n+1},\mathcal{O}
_ {X_{n+1}})$ and $x \neq 0$ we have $h_{(a,b)}(i(x)) \neq i(x)$. So, we're done.

\end{proof}

Now we'll take a look at $H^1_{log_fl}(X_n, \mu_p)$. First we prove the following lemma which we will be used in the computation
\begin{lem}
Let $X_{0} = \operatorname{Proj}(k[x,y,z]/(x+y-z))$. Let log structure on $X_{0}$ be the log structure associated with $D=\{ 0, 1, \infty \}$. Then  $H^1_{log_fl}(X_{0}, \mu_p) = \mathbb{Z}/p\mathbb{Z} \oplus \mathbb{Z}/p\mathbb{Z} $.
\end{lem}

\begin{proof}
We start with the following exact sequence in the log-flat topology ( Proposition 4.2 in \cite{kato2019logarithmic}).

\begin{equation*}
    0 \to \mu_p \to \mathbb{G}_{m,\log} \to \mathbb{G}_{m,\log} \to 0.
\end{equation*}

This gives us the following long exact sequence
\begin{equation}
\label{kummerlogles}
    0 \to H^0_{log_fl}(X_{0}, \mu_p) \to H^0_{log_fl}(X_{0}, \mathbb{G}_{m,\log}) \to  H^0_{log_fl}(X_{0}, \mathbb{G}_{m,\log}) \to  H^1_{log_fl}(X_{0}, \mu_p) \to H^1_{log_fl}(X_{0}, \mathbb{G}_{m,\log})
\end{equation}

Next we compute $H^1_{log_fl}(X_{0}, \mathbb{G}_{m,\log})$. Using Theorem 5.1 in \cite{kato2019logarithmic}, we get that $H^1_{log_fl}(X_{0}, \mathbb{G}_{m,\log}) = H^1_{fl}(X_{0}, \mathbb{G}_{m,\log}) $.

To compute $ H^1_{fl}(X_{0}, \mathbb{G}_{m,\log})$ we start with the following exact sequence in the flat topology

\begin{equation}
    0 \to \mathbb{G}_{m} \to \mathbb{G}_{m,\log} \to \oplus_{x \in \{0.1, \infty\}} i_{{x}_{*}} \mathbb{Z}
\end{equation}

Here $\mathbb{Z}$ is the constant sheaf on $x$ and $i_{x}: x \xhookrightarrow{} X_{0}$. From the above sequence we get the long exact sequence
\begin{align*}
   & H^{0}_{fl}(X_{0},\mathbb{G}_{m}) \to H^{0}_{fl}(X_{0}, \mathbb{G}_{m,\log}) \to H^{0}_{fl}(X_0, \oplus_{x \in \{0.1, \infty\}} i_{{x}_{*}} \mathbb{Z}) \to  H^{1}_{fl}(X_{0},\mathbb{G}_{m}) \to H^{1}_{fl}(X_{0}, \mathbb{G}_{m,\log}) \\ \to & H^{1}_{fl}(X_{0},\oplus_{x \in \{0.1, \infty\}} i_{{x}_{*}} \mathbb{Z}).  
\end{align*}

Now $H^{1}_{fl}(X_{0},\oplus_{x \in \{0.1, \infty\}} i_{{x}_{*}} \mathbb{Z}) = \oplus_{x \in 
\{0,1, \infty\}} H^{1}_{fl}(X_{0},i_{{x}_{*}} \mathbb{Z})$. Since, $H^{1}_{fl}(x, 
\mathbb{Z}) = 0 $ we get $H^{1}_{fl}(X_{0},\oplus_{x \in \{0.1, \infty\}} i_{{x}_{*}} \mathbb{Z}) = 0$. Now the map 
from $H^{0}_{fl}(X_0, \oplus_{x \in \{0.1, \infty\}} i_{{x}_{*}} \mathbb{Z}) = \mathbb{Z} \oplus \mathbb{Z}   \oplus 
\mathbb{Z} \to  H^{1}_{fl}(X_{0},\mathbb{G}_{m}) = \mathbb{Z}$ is the sum map therefore surjective. Thus it follows 
$H^{1}_{fl}(X_{0}, \mathbb{G}_{m,\log}) = 0$. Also note that $H^{0}_{fl}(X_{0}, \mathbb{G}_{m,\log})/ 
H^{0}_{fl}(X_{0},\mathbb{G}_{m}) = \mathbb{Z} \oplus \mathbb{Z}$. Now from 
\ref{kummerlogles}, we get

$ 0 \to H^0_{log_fl}(X_{0}, \mu_p) \to H^0_{log_fl}(X_{0}, \mathbb{G}_{m,\log}) \to  H^0_{log_fl}(X_{0}, \mathbb{G}_{m,\log}) \to  H^1_{log_fl}(X_{0}, \mu_p) \to 0$

Since $H^{0}_{fl}(X_{0},\mathbb{G}_{m}) = k^{\times}$ and $H^{0}_{fl}(X_{0}, \mathbb{G}_{m,\log})/ H^{0}_{fl}(X_{0},\mathbb{G}_{m}) = \mathbb{Z} \oplus \mathbb{Z}$, we get 
$H^1_{log_fl}(X_{0}, \mu_p) = \mathbb{Z}/p\mathbb{Z} \oplus \mathbb{Z}/p\mathbb{Z}$. 

\end{proof}

Using the Leray Spectral sequence \ref{leray} we get the following exact sequence

\begin{equation} \label{leray}
0 \to H_{fl}^1(X_n, \mu_p ) \to H^1_{log_fl}(X_n, \mu_p) \to H_{fl}^0(X_n, R^{1}\epsilon_{*}(\mu_p) ) 
\end{equation}

Also, let $X_{0} = \operatorname{Proj}(k[x,y,z]/(x+y-z))$.
Consider the map $k[x,y,z]/(x+y-z)^{p^n} \to k[x,y,z]/(x+y-z)$ where $x \to x$, $y \to y$ and $z \to z$. Thus we get a map $j:X_0 \to X_n$. Using the Leray spectral sequence, we get the following commutative diagram

\begin{tikzcd}
H_{fl}^1(X_n, \mu_p )  \arrow[r] \arrow[d,"j"] &H^1_{log_fl}(X_n, \mu_p)  \arrow[r] \arrow[d,"j"] & H_{fl}^0(X_n, R^{1}\epsilon_{*}(\mu_p) ) \arrow[d,"j"]\\
H_{fl}^1(X_0, \mu_p )  \arrow[r]  & H^1_{log_fl}(X_0, \mu_p)  \arrow[r]  & H_{fl}^0(X_0, R^{1}\epsilon_{*}(\mu_p) ) 
\end{tikzcd}

Using Theorem \ref{thmr} we can compute  $H_{fl}^0(X_n, R^{1}\epsilon_{*}(\mu_p) )$ and $H_{fl}^0(X_0, R^{1}\epsilon_{*}(\mu_p) )$ .

\begin{equation}
    R^{1}\epsilon_{*}(\mu_p) = \underset{n \neq 0}{\varinjlim} {\Hom}(\mu_{n}, \mu_p) \otimes_{\mathbb{Z}} (\mathbb{G}_{m,\log} /\mathbb{G}_m ),
\end{equation}

To compute $H_{fl}^0(X_n, \underset{n \neq 0}{\varinjlim} {\Hom}(\mu_{n}, \mu_p) \otimes_{\mathbb{Z}} (G_{m,\log} 
/G_m ) )$, first note that $\mathbb{G}_{m,\log} /\mathbb{G}_m = \bigoplus_{i \in \{1,0, \infty\}} \mathbb{Z}_i$ . 
Now, $\underset{n \neq 0}{\varinjlim} {\Hom}(\mu_{n}, \mu_p) \otimes_{\mathbb{Z}}  (\bigoplus_{i \in \{1,0, \infty\}} \mathbb{Z}_i) = {\Hom}(\mu_p, \mu_p)_{ \otimes_{\mathbb{Z}}} (\bigoplus_{i \in \{1,0, \infty\}} \mathbb{Z}_i)$. Now we 
get ${\Hom}(\mu_p, \mu_p) = {\Hom}(\mu_p, \mathbb{G}_m) = \mathbb{Z}/p\mathbb{Z}$. Therefore, we get $H_{fl}^0(X_n, 
\underset{n \neq 0}{\varinjlim} {\Hom}(\mu_{n}, \mu_p) \otimes_{\mathbb{Z}} (G_{m,\log} /G_m ) ) = \mathbb{Z}/p 
\mathbb{Z} \oplus \mathbb{Z}/p \mathbb{Z} \oplus \mathbb{Z}/p \mathbb{Z}$. Note that the map from $j:H_{fl}^0(X_0, 
R^{1}\epsilon_{*}(\mu_p) ) \to H_{fl}^0(X_0, R^{1}\epsilon_{*}(\mu_p) )$ is an isomorphism. So, we the following commutative diagram .

\begin{tikzcd}
\label{mulog}
H_{fl}^1(X_n, \mu_p )  \arrow[r] \arrow[d,"j"] &H^1_{log_fl}(X_n, \mu_p)  \arrow[r] \arrow[d,"j"] &  \mathbb{Z}/p \mathbb{Z} \oplus \mathbb{Z}/p \mathbb{Z} \oplus \mathbb{Z}/p \mathbb{Z} \arrow[d,"j"]\\
H_{fl}^1(X_0, \mu_p )  \arrow[r]  & H^1_{log_fl}(X_0, \mu_p)  \arrow[r]  &  \mathbb{Z}/p \mathbb{Z} \oplus \mathbb{Z}/p \mathbb{Z} \oplus \mathbb{Z}/p \mathbb{Z}
\end{tikzcd}

Let $I = H_{fl}^1(X_n, \mu_p ) $. Now let $x \in \mathbb{Z}/p \mathbb{Z} \oplus \mathbb{Z}/p \mathbb{Z} \subseteq H^1_{log_fl}(X_n, \mu_p)$ then for all $(a,b) \in \mu_{p^n}(R) \times \mu_{p^n}(R)$ and for all $R$, $h_{(a,b)}(i(x)) =i(x)$. Here $i$ is the canonical map $i:H_{fl}^1(X_n, \mu_p ) \to H_{fl}^1(X_n \otimes_{k}R, \mu_p )$. Therefore from Proposition $\ref{mufixed}$ we get $I \cap \mathbb{Z}/p \mathbb{Z} \oplus \mathbb{Z}/p \mathbb{Z} = 0$. We claim that  $H^1_{log_fl}(X_n, \mu_p) = I \oplus \mathbb{Z}/p \mathbb{Z} \oplus \mathbb{Z}/p \mathbb{Z} $. Since, $H_{fl}^1(X_0, \mu_p ) = 0$ and $H^1_{log_fl}(X_0, \mu_p)= \mathbb{Z}/p \mathbb{Z} \oplus \mathbb{Z}/p \mathbb{Z}$  the image of $H^1_{log_fl}(X_n, \mu_p)$ in $\mathbb{Z}/p \mathbb{Z} \oplus \mathbb{Z}/p \mathbb{Z} \oplus \mathbb{Z}/p \mathbb{Z}$ isomorphic to $ \mathbb{Z}/p \mathbb{Z} \oplus \mathbb{Z}/p \mathbb{Z}$. This follows from the commutativity of the diagram \ref{mulog}. But since $I \oplus \mathbb{Z}/p \mathbb{Z} \oplus \mathbb{Z}/p \mathbb{Z} $ is contained in $H^1_{log_fl}(X_n, \mu_p)$ and  $H^1_{log_fl}(X_n, \mu_p)/I \cong  \mathbb{Z}/p \mathbb{Z} \oplus \mathbb{Z}/p \mathbb{Z}$ ,  we get $H^1_{log_fl}(X_n, \mu_p) = I \oplus \mathbb{Z}/p \mathbb{Z} \oplus \mathbb{Z}/p \mathbb{Z} $.

From the above we get the following proposition

\begin{prop}
\label{mulogfixed}
Let $R$ be a $k$-algebra $R$ then we have a map $i :H^1_{fl}(X_n, \mu_p) \to H^1_{logfl}(X_n \otimes_k R, \mu_p)$. Also for every $(a,b) \in \mu_{p^n}(R) \times \mu_{p^n}(R)$, $h_{(a,b)}$ induces map $h_{(a,b)}:  H^1_{logfl}(X_n \otimes_k R, \mu_p) \to H^1_{logfl}(X_n \otimes_k R, \mu_p) $.

Let $H^1_{logfl}(X_n, \mu_p)^{\mu_{p^n}(R) \times \mu_{p^n}(R)} = \{ x \in H^1_{fl}(X_n, \mu_p):\text{for all } (a,b) \in \mu_{p^n}(R) \times \mu_{p^n}(R)\text{ ,} h_{(a,b)}(i(x)) =i(x) \}$. There exists a $k$-algebra $R$ such that 
\begin{equation*}
H^1_{logfl}(X_n, \mu_p)^{\mu_{p^n}(R) \times \mu_{p^n}(R)} = \mathbb{Z}/p \mathbb{Z} \oplus \mathbb{Z}/p \mathbb{Z}
\end{equation*}
\end{prop}

Now combining Proposition \ref{zpzfixedlog}, \ref{mulogfixed} and \ref{alphafixedlog} gives us Theorem \ref{thm::main2}. 

\bibliography{log_scheme}
\bibliographystyle{amsplain}

\end{document}